\documentclass[11pt]{amsart}
\usepackage{amsmath}
\usepackage{tikz}
\usepackage{enumerate, comment}
\usepackage{systeme}

\usepackage{color}
\title{Some $q$-analogues of multinomial coefficients}
\newtheorem{theorem}{Theorem}
\newtheorem{conjecture}[theorem]{Conjecture}
\newtheorem{lemma}{Lemma}
\newtheorem{proposition}[theorem]{Proposition}
\newtheorem{corollary}[theorem]{Corollary}
\theoremstyle{definition}
\newtheorem*{problem*}{Problem}

\theoremstyle{remark}
\newtheorem{remark}{Remark}

\newcommand{\qrfac}[2]{{\left({#1}; q\right)_{#2}}} 
\newcommand{\qbin}[2]{\begin{bmatrix} #1 \\ #2 \end{bmatrix}_q}
\newcommand{\pqbin}[3]{\begin{bmatrix} #1 \\ #2 \end{bmatrix}_{#3}}
\newcommand{\Fq}{\mathbb{F}_q}

\newcommand{\pqrfac}[3]{{\left({#1};#3\right)_{#2}}}
\numberwithin{equation}{section}

\makeatletter

\@ifundefined{qbinom}{
  \newcommand{\qbinom}[2]{\begin{bmatrix} #1 \\ #2 \end{bmatrix}_q}
}{}

\@ifundefined{qrfac}{
  \newcommand{\qrfac}[3]{(#1;#3)_{#2}}
}{}

\@ifundefined{ZZ}{
  
}{}

\makeatother

\newcommand{\jj}{\boldsymbol j}

\newcommand\sumj{{\left| \boldsymbol {j} \right|}}
\newcommand\sumk{{\left| \boldsymbol {k} \right|}}
\newcommand{\sumbinomial}[2]{\sum\limits_{r=1}^{#2} {\binom{#1_r}{2}}}

\allowdisplaybreaks

\usepackage{color}
 \usepackage[bookmarks=true]{hyperref}
\hypersetup{
    colorlinks=true, 
    linktoc=all,     
    linkcolor=blue,  
    citecolor=blue,
    hyperindex =true
}

\renewcommand{\j}{\boldsymbol {j}}

\author[G.~Bhatnagar]{Gaurav Bhatnagar
}

\address{RamanujanExplained.org, 18 Chitra Vihar, Delhi 110092}
\email{bhatnagarg@gmail.com}

\author[A.~Prasad]{Amritanshu Prasad}
\address{Homi Bhabha National Institute, Anushakti Nagar, Mumbai, 400094 and
The Institute of Mathematical Sciences, Chennai 600113}
\email{amri@imsc.res.in}

\title[An extension of a $q$-binomial theorem]{A bibasic double sum extension of a $q$-binomial theorem arising out of subspace enumeration}



\begin{document}
\maketitle

\begin{abstract} 
We prove a conjecture that arose in the context of a subspace enumeration problem over finite fields. We prove, more generally, a bibasic, double-sum identity, which extends a $q$-analogue of the (terminating) binomial theorem. 
\end{abstract}

\section{Introduction}
We prove a conjecture that arose in a combinatorial context, while considering a subspace enumeration problem over finite fields. Our proof uses a $q$-analogue of the terminating binomial theorem, and leads to a double-sum identity which extends this $q$-binomial theorem. This conjecture belongs to a broader set of (conjectural) combinatorial identities.

We require the notation of $q$-rising factorials from Gasper and Rahman~\cite{GR90}. 
The {\em $q$-shifted factorials}, for $k$ a non-negative integer, are defined as
\begin{gather*}
\qrfac{a}{k} :=\prod\limits_{j=0}^{k-1} \big(1-aq^j\big);
\end{gather*}
and for $|q|<1$,
\begin{gather*}
\qrfac{a}{\infty} := \prod\limits_{j=0}^{\infty} \big(1-aq^j\big).
\end{gather*}
The parameter $q$ is called the {\em base}. Our main result is bibasic---it has two independent bases $p$ and $q$.

We use the notation $\jj=(j_1, j_2, \dots, j_r)$ and $\sumj$ for $j_1+j_2+j_3+\cdots+j_r$. 
 The $q$-multinomial coefficients are defined as:
\begin{displaymath}
    \qbinom{n}{j_1,j_2,j_3, \dots, j_r} = \frac{(q;q)_n}{(q;q)_{j_1} (q;q)_{j_2} \cdots (q;q)_{j_r}\qrfac{q}{n-\sumj}}.
\end{displaymath}
These are polynomials in $q$. As is common, we use the convention:
\begin{equation}\label{natural-term}
\frac{1}{\qrfac{q}{n}} = 0 \text{ if $n<0$},
\end{equation}
so the $q$-multinomial coefficient is $0$ if $j_r<0$ for any $r$, or, if $n< \sumj$. 
When $r=1$, this is called the $q$-binomial coefficient. 

The main objective of this paper is to prove the following conjecture. 
\begin{conjecture}[Prasad and Ram]\label{conj:1} Let  $q$ be a prime power and $\omega$ denote a primitive cube root of unity; let $\j = (j_1,j_2,j_3)$. Then:
\begin{multline}\label{conj1}
\sum_{\sumj = 3m} 
  q^{
  \frac{1}{3}\big(
  \sum\limits_{i=1}^{3}
     \binom{j_i}{2}
    - g(j_1,j_2,j_3)
  \big)
  }
  \qbin{3m}{j_1,j_2,j_3}
  \omega^{\,j_2 + 2j_3}
  = \delta_{m,0} =
  \begin{cases}
  1, & \text{if $m=0$};\\
  0, & \text{otherwise},
  \end{cases} \\[6pt]
\text{where}\quad
g(j_1,j_2,j_3)
  =
  \begin{cases}
    2j_2 + j_3, & \text{if } j_2+ 2j_3 \equiv 2 \pmod{3};\\
    j_2 + 2j_3, & \text{if } j_2 + 2j_3 \equiv 0,1 \pmod{3}.
  \end{cases}
\end{multline}
\end{conjecture}
Note that the restriction of $q$ being a prime power can be removed, since the left hand side is a polynomial in $q$. It is easy to see that
$$
  \sum\limits_{i=1}^{3}
     \binom{j_i}{2}
    - g(j_1,j_2,j_3)
  $$
is a non-negative integer divisible by $3$, so the power of $q$ is a non-negative integer. 
So this is a polynomial identity and $q$ can be an arbitrary complex number. Alternatively, all the identities in this paper may be regarded as identities for formal power series with variable $q$. 

Prasad and Ram arrived at this (unpublished) conjecture while studying an enumeration problem over finite fields posed by Bender, Coley, Robbins, and Rumsey~\cite{MR1141317}.
As a follow-up to their work~\cite{MR4555237, MR4797454}, they were seeking the coefficients of a \emph{universal formula} of the form
\begin{displaymath}
   \sum_{\sumj =3m} c^T_{\j}(q) Y^T_{\j} = \sigma^T_{(m,m,m)},
\end{displaymath}
for any matrix $T\in M_{3m}(\Fq)$. Here $\sigma^T_{(m,m,m)}$ is the number of $m$-dimensional subspaces $W\subset \Fq^{3m}$ such that $W\oplus TW\oplus T^2W = \Fq^{3m}$; $Y^T_{\j}$ is the number of flags $W_1\subset W_2\subset W_3$ of subspaces of $\Fq^{3m}$ such that $T(W_i)\subset W_i$, with $\dim(W_i/W_{i-1})=j_i$ for $i=1,2,3$; and $c^T_{\j}(q)$ is a polynomial in $q$ independent of the matrix $T$. The form of the universal formula was predicted by \cite{MR1141317}. We give more details in \S\ref{sec:2}.


We prove the following generalization of \eqref{conj1}.
\begin{theorem}\label{th:conj1-gen}  
Let $q$ be a complex number, $\omega$ denote a primitive cube root of unity; and let $\j = (j_1,j_2,j_3)$.
Then
\begin{multline}\label{conj1-gen}
\sum_{\sumj = 3m} 
  q^{
  \frac{1}{3}\big( j_2^2+j_2j_3+j_3^2
    - g(j_1,j_2,j_3)
  \big)
  }
  z^{j_2+j_3}
  \qbin{3m}{j_1,j_2,j_3}
    \omega^{\,j_2 + 2j_3}
  = \qrfac{z}{3m}, \\
\text{where,}\quad
g(j_1,j_2,j_3)
  =
  \begin{cases}
    2j_2 + j_3, & \text{if } j_2+ 2j_3 \equiv 2 \pmod{3};\\
    j_2 + 2j_3, & \text{if } j_2 + 2j_3 \equiv 0,1 \pmod{3}.
  \end{cases}
\end{multline}
\end{theorem}
Conjecture \ref{conj:1} follows. When $m=0$, both sides of \eqref{conj1-gen} are $1$. This matches \eqref{conj1} for $m=0$. For $m>0$, take  $z=q^{-m}$ in \eqref{conj1-gen}. The right-hand side reduces to $0$; this shows \eqref{conj1} for $m>0$.

Alternatively, we can write \eqref{conj1-gen} as:
\begin{equation}\label{eq8-gen}
\sum_{n=0}^{3m}
q^{n^2} z^n
\sum_{k=n}^{2n}
 q^{\frac{1}{3}( {k^2}-g(k,n))-kn}
 \qbin{3m}{2n-k, k-n}
  \omega^k
  =
  \qrfac{z}{3m},
\end{equation}
where,
 $$ 
 g(k,n)
  =
  \begin{cases}
    3n-k, & \text{if }k \equiv 2 \pmod{3};\\
    k, & \text{if }k \equiv 0,1 \pmod{3}.
  \end{cases}
$$

The proof yields an even more general, bibasic, result:
\begin{theorem}\label{th:conj1-gen2} 
Let $p$ and $q$ be complex numbers. Then
\begin{equation}\label{eq8-gen2}
\sum_{n=0}^{3m}
q^{\binom{n+1}{2}} p^{\binom{n}{2}} z^n \pqbin{3m}{n}{p}
\sum_{k=n}^{2n}
 q^{\frac{1}{3}( {k^2}-g(k,n))-kn}
\qbin{n}{k-n}
  \omega^k
  =\pqrfac{z}{3m}{p},
\end{equation}
where,
 $$ g(k,n)
  =
  \begin{cases}
    3n-k, & \text{if }k \equiv 2 \pmod{3};\\
    k, & \text{if }k \equiv 0,1 \pmod{3}.
  \end{cases}
$$
\end{theorem}

Conjecture~\ref{conj1} can be regarded as a $2$-dimensional sum extending a $q$-analogue of the binomial theorem. 

An analogous $1$-dimensional case of the conjecture is the identity
\begin{equation}\label{conj-1d}
    \sum_{j=0}^{2m} (-1)^j q^{ 
    \binom{j}{2} -mj}
\qbin{2m}{j} = \delta_{m,0},
\end{equation}
which is obtained by taking $n=2m$ and $z=q^{-m}$ in the following $q$-analog of the (terminating) binomial theorem~\cite[Ex.\ 1.2(vi)]{GR90}:
\begin{equation}\label{q-vandermonde1a}
\sum_{j=0}^{n}
  \qbin{n}{j}
    (-1)^j\, q^{\binom{j}{2}} z^j   = \pqrfac{z}{n}{q}.
\end{equation}
Interestingly, the proof of the $2$-dimensional sum also relies upon the very same identity. 

Taking $m\to\infty$ in Theorem~\ref{th:conj1-gen2}, we obtain the following bibasic, double-summation identity.
\begin{corollary} Let $|p|<1$ and $g(k,n)$ be as in \eqref{eq8-gen2}. Then
\begin{equation}\label{cor1}
\sum_{ n,k\ge 0} 
p^{\binom{n}{2}}
q^{\binom{n+1}{2}+\frac{1}{3}( {k^2}-g(k,n))-kn}  
\frac{\qrfac{q}{n}  z^n \omega^k  }{\pqrfac{p}{n}{p} \pqrfac{q}{k-n}{q}\pqrfac{q}{2n-k}{q}} 
   =\pqrfac{z}{\infty}{p}.
\end{equation}
\end{corollary}
Note that  in the sum on the left-hand side of \eqref{cor1}, the index $k$ ranges from $n$ to $2n$. 

When $p=q$ in \eqref{cor1}, we obtain
\begin{equation}\label{cor1-a}
\sum_{ n=0}^\infty 
q^{n^2}
z^n
\sum_{k=n}^{2n}
q^{\frac{1}{3}( {k^2}-g(k,n))-kn}  
\frac{ \omega^k}{ \pqrfac{q}{k-n}{q}\pqrfac{q}{2n-k}{q}} 
   =\pqrfac{z}{\infty}{q}.
\end{equation}
The outer sum (if we ignore the inner sum with index $k$) is what 
Andrews~\cite{Andrews1981} called a  partial theta function. 

Conjecture~\ref{conj:1} is a special case of a broader combinatorial conjecture. In \S\ref{sec:2}, we explain the origin of \eqref{conj1}, as well as the broader conjecture,
and the link to a classic combinatorial formula---namely the Touchard--Riordan~\cite{MR366686, MR46325} formula. We give a proof of Theorems~\ref{th:conj1-gen} and \ref{th:conj1-gen2} in \S\ref{sec:3}.  Our proof is just the first step in understanding these combinatorial identities.


\section{Counting subspaces and Conjecture~\ref{conj:1}}\label{sec:2}
In this section we explain the origins of Conjecture 1. This conjecture is a special case of a more general conjecture (Conjecture 5, below), 
which arose from the work of Prasad and Ram~\cite{MR4555237,  MR4797454} on an enumerative problem posed by Bender, Coley, Robbins and Rumsey~\cite{MR1141317}. A proof of Conjecture 5 would lead to an extension of the classical Touchard--Riordan formula.

Let $q$ be a prime power, and let $\Fq$ denote a finite field with $q$ elements.
Consider a matrix $T\in M_n(\Fq)$ and a subspace $W$ of $\Fq^n$.
Let  $W^{(0)}=\{0\}$, and for each $k\geq 1$, let
\begin{displaymath}
  W^{(k)} = W + T(W) + \cdots + T^k(W).
\end{displaymath}
Let $j_i = \dim W^{(i)}/W^{(i-1)}$ for $i\geq 1$.
The sequence $(j_1,j_2,\dots)$ is called the {\em dimension sequence} of $W$ with respect to $T$.
It is not difficult to see that if $(j_1,j_2,\dots)$ is the dimension sequence of a subspace $W\subset \Fq^n$ with respect to a matrix $T\in M_n(\Fq)$, then $n\geq j_1\geq j_2\geq \cdots$ and $j_1+j_2+\cdots \leq n$.
In other words, the dimension sequence of a subspace is a partition of some integer less than or equal to $n$.

Bender, Coley, Robbins and Rumsey~\cite{MR1141317} considered the following enumerative problem:
\begin{problem*}
  Let $T\in M_n(\mathbf F_q)$ be any matrix over the finite field $\mathbf F_q$ with $q$ elements.
  Enumerate the number $\sigma^T_{(j_1,j_2,\dots)}$ of subspaces $W$ with dimension sequence $(j_1,j_2,\dots)$.
\end{problem*}
The solution to this problem depends on the matrix $T$.
If $(j_1,j_2,\dots,j_k)$ is a sequence of non-negative integers such that $j_1+j_2+\cdots +j_k \leq n$, let
\begin{displaymath}
  Y^T_{(j_1,j_2,\dots,j_k)} = \#\{W_1\subset W_2\subset\dotsb \subset W_k\mid T(W_i)\subset W_i,\; \dim(W_i/W_{i-1})=j_i\}.
\end{displaymath}
The number $Y^T_{(j_1,j_2,\dots,j_k)}$ remains unchanged under reordering of the sequence $(j_1$, $j_2$, $\dots$, $j_k)$, so we may assume that 
$$j_1 \ge j_2 \ge \dots \ge j_k$$ is a partition of some integer less than or equal to $n$.
Bender, Coley, Robbins and Rumsey~\cite{MR1141317} showed that there exists a \emph{universal formula}
\begin{equation}\label{eq:universal}
  \sigma^T_{(j_1,j_2,\dots)} = \sum_{(k_1,k_2,\dots)} c^{(k_1,k_2,\dots)}_{(j_1,j_2,\dots)}(q) Y^T_{(k_1,k_2,\dots)},
\end{equation}
for all partitions $(j_1,j_2,\dots)$, where the coefficients $c^{(k_1,k_2,\dots)}_{(j_1,j_2,\dots)}$ are polynomials in $q$ independent of the matrix $T$.

The general problem of determining the coefficients $c^{(k_1,k_2,\dots)}_{(j_1,j_2,\dots)}$ was solved by Ram~\cite{ram2024subspaceprofilesfinitefields}, where the solution is given in terms of symmetric functions. However, the work in this paper concerns more explicit formulas solving special cases of this problem. 

For example, when $n=2m$, $\sigma^T_{(m,m)}$ counts the number of subspaces $W\subset \Fq^{2m}$ of dimension $m$ such that $W\oplus TW = \Fq^{2m}$.
Prasad and Ram~\cite{MR4555237} obtained
\begin{equation}\label{eq:mm}
  \sigma^T_{(m,m)} = q^{\binom m2}\sum_{j=0}^{2m} (-1)^j q^{\binom{m-j+1}{2}} Y^T_{(j)},
\end{equation}
giving a closed formula for the coefficients in \eqref{eq:universal} in the case of the partition $(m,m)$.
When $T=0$ (or any scalar matrix), $$Y^T_{(j)}=\qbin{2m}{j},$$ 
while $\sigma^T_{(m,m)}=\delta_{m,0}$ (since $W=TW$ for any subspace $W$). As a result, we obtain \eqref{conj-1d}.

In seeking a formula analogous to \eqref{eq:mm} for $\sigma^T_{(m,m,m)}$, Prasad and Ram arrived at the following (unpublished) conjecture.  
\begin{conjecture}[Prasad and Ram]\label{conj:2} For any $T\in M_{3m}(\Fq)$:
\begin{equation}\label{eq:conj2}
  \sigma^T_{(m,m,m)} = q^{m^2-m}\sum_{\sumj=3m} q^{\frac 13\left(\sum_{i=1}^3 \binom{j_i}{2} - g(j_1,j_2,j_3)\right)} Y^T_{\j} \omega^{j_2+2j_3},
\end{equation}
with $g(j_1,j_2,j_3)$ as in \eqref{conj1}.
\end{conjecture}
Again, when $T$ is a scalar matrix, 
$$Y^T_{\j} = \qbin{3m}{j_1,j_2,j_3},$$ 
and $\sigma^T_{(m,m,m)}=\delta_{m,0}$. This yields Conjecture~\ref{conj:1}, the proof of which is the object of this paper, and given in \S\ref{sec:3}.

By varying $T$, a plethora of new identities can be obtained from Conjecture~\ref{conj:2}. 
We illustrate one such possibility by giving an example of a combinatorial nature.

When $T\in M_{2m}(\Fq)$ has distinct eigenvalues in $\Fq$, Prasad and Ram \cite{MR4797454} obtained a different formula for the left-hand side  $\sigma^T_{(m,m)}$ of 
\eqref{eq:mm}, in terms of chord diagrams:
\begin{equation}\label{eq:chord-mm}
  \sigma^T_{(m,m)} = (q-1)^m \sum_{\sigma} q^{v(\sigma)} Y^T_{(2m)},  
\end{equation}
the sum being over all chord diagrams $\sigma$ on $2m$ nodes, and $v(\sigma)$ being the number of crossings in $\sigma$.
Figure~\ref{fig:chord-example} is an example of a chord diagram on eight nodes with two crossings.
\begin{figure}[h]
\centering
\begin{tikzpicture}
  [every node/.style={circle,fill=black,inner sep=0pt, minimum size=6pt}]
  \node[label=below:$1$] (1) at (1,0) {};
  \node[label=below:$2$] (2) at (2,0) {};
  \node[label=below:$3$] (3) at (3,0) {};
  \node[label=below:$4$] (4) at (4,0) {};
  \node[label=below:$5$] (5) at (5,0) {};
  \node[label=below:$6$] (6) at (6,0) {};
  \node[label=below:$7$] (7) at (7,0) {};
  \node[label=below:$8$] (8) at (8,0) {};
  \draw[thick,color=teal]
  (1) [out=45, in=135] to (4);
  \draw[thick,color=teal]
  (2) [out=45, in=135] to (6);
  \draw[thick,color=teal]
  (3) [out=45, in=135] to (5);
  \draw[thick,color=teal]
  (7) [out=45, in=135] to (8);
\end{tikzpicture}
\caption{A chord diagram on eight nodes with two crossings.}
\label{fig:chord-example}
\end{figure}
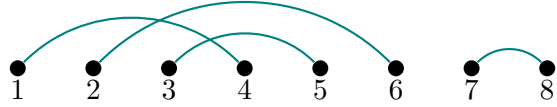
Combining \eqref{eq:mm} and \eqref{eq:chord-mm} gives a new proof of the Touchard--Riordan formula~\cite{MR366686, MR46325} (see also~\cite[pp.~337--344]{MR2339282} for an elegant exposition):
\begin{equation}\label{eq:touchard-riordan}
  (q-1)^m \sum_{\sigma} q^{v(\sigma)} = \sum_{j=0}^{2m} (-1)^j q^{\binom{m-j+1}{2}} \binom{2m}{j}.
\end{equation}

Similarly, when $T\in M_{3m}(\Fq)$ has distinct eigenvalues in $\Fq$, $\sigma^T_{(m,m,m)}$ can be expressed in terms of set partitions of $[3m]$ into $m$ blocks of size $3$. For such a set partition, an interlacing of blocks $(a,b,c)$ and $(x,y,z)$ is either $a<x<b<y$ or $b<y<c<z$.
The results of~\cite{MR4797454} show that
\begin{equation}
  \sigma^T_{(m,m,m)} = (q-1)^{2m} \sum_{\pi} q^{\text{interlacings}(\pi)} Y^T_{(3m)},
\end{equation}
the sum being over all set partitions $\pi$ of $[3m]$ into $m$ blocks of size $3$, and $\text{interlacings}(\pi)$ being the number of interlacings in $\pi$.
In this case, 
$$Y^T_{\j}=\binom{3m}{j_1,j_2,j_3}$$ is the ordinary multinomial coefficient, and so  \eqref{eq:conj2} implies the following extension of the Touchard--Riordan formula:
\begin{equation}\label{eq:touchard-riordan-extension}
  (q-1)^{2m} \sum_{\pi} q^{\text{interlacings}(\pi)} = q^{m^2-m}\sum_{\sumj =3m} q^{\frac 13\left(\sum_{i=1}^3 \binom{j_i}{2} - g(j_1,j_2,j_3)\right)} \binom{3m}{j_1,j_2,j_3} \omega^{j_2+2j_3},
\end{equation}
with $g(j_1,j_2,j_3)$ as in \eqref{conj1}.
This is currently unproven.
\section{Proof of Theorems~\ref{th:conj1-gen} and \ref{th:conj1-gen2}}\label{sec:3}

In this section, we prove Theorem~\ref{th:conj1-gen}, and then  extend the proof  to obtain Theorem~\ref{th:conj1-gen2}. 


To begin, we  first take $j_1=3m -j_2-j_3$ in the sum on the left-hand side of \eqref{conj1-gen}.  The sum becomes over $j_2+j_3\le 3m$. The sum may then be written in the form
$$\sum_{j_2+j_3\le 3m} (***) = \sum_{n=0}^{3m} \sum_{j_2+j_3=n} (***).$$

Next, we change indices in the inner sum to $k$. We see that:
$$
\systeme{j_2 + 2j_3 = k , j_2 + j_3 = n}
\quad\iff\quad
\systeme{j_3 = k-n,  j_2 = 2n-k}.
$$
Note that $2j_2+j_3 = 3n-k$. Further $0\le j_3 \le n$, so $n\le k\le 2n$. 
We obtain:
\begin{equation}\label{eq8}
\sum_{n=0}^{3m}\sum_{k=n}^{2n}
q^{n^2} z^n q^{\frac{1}{3}( {k^2}-g(k,n))-kn}
  \frac{
    \pqrfac{q}{3m}{q}
  }{
    \pqrfac{q}{3m-n}{q}
    \pqrfac{q}{2n-k}{q}
    \pqrfac{q}{k-n}{q}
  }
  \omega^k,
\end{equation}
 where, now, 
 $$ 
 g(k,n)
  =
  \begin{cases}
    3n-k, & \text{if }k \equiv 2 \pmod{3};\\
    k, & \text{if }k \equiv 0,1 \pmod{3}.
  \end{cases}
$$
Note that the sums terminate naturally, due to \eqref{natural-term}.
So $n\le 3m, k\le 2n$, and $k\ge n$ is implied by the terms; thus we may as well write the sum with an unrestricted index. We now break the sums into three parts, where $k\mapsto 3k, 3k+1, \text{ and }3k-1$. 
We obtain three sums in the form:
$$\sum_{n=0}^\infty\sum_{k=0}^\infty A(n,k)
= \sum_{n=0}^\infty\sum_{k=0}^\infty A(n,3k)
+ \sum_{n=0}^\infty\sum_{k=0}^\infty A(n,3k+1)
+ \sum_{n=0}^\infty\sum_{k=0}^\infty A(n,3k-1). 
$$
Eventually, we get that the left-hand side equals the following:

\begin{multline}\label{eqn9}
\sum_{n=0}^{\infty}
  q^{n^2 } z^n
  \frac{
    \pqrfac{q}{3m}{q}
  }{
    \pqrfac{q}{3m-n}{q}\,
    \pqrfac{q}{n}{q}
  }
  \sum_{k=0}^{\infty}
  q^{3k^2 - k - 3kn}
  \frac{
    \pqrfac{q}{n}{q}
  }{
    \pqrfac{q}{2n-3k}{q}\,
    \pqrfac{q}{3k-n}{q}
  }\\
  +\omega \sum_{n=0}^{\infty} 
  q^{n^2 } z^n
  \frac{
    \pqrfac{q}{3m}{q}
  }{
    \pqrfac{q}{3m-n}{q}\,
    \pqrfac{q}{n}{q}
  }
  \sum_{k=0}^{\infty}
  q^{3k^2 + k - 3kn-n}
  \frac{
    \pqrfac{q}{n}{q}
  }{
    \pqrfac{q}{2n-3k-1}{q}\,
    \pqrfac{q}{3k-n+1}{q}
  }\\
+\omega^2 \sum_{n=0}^{\infty} 
  q^{n^2 } z^n
  \frac{
    \pqrfac{q}{3m}{q}
  }{
    \pqrfac{q}{3m-n}{q}\,
    \pqrfac{q}{n}{q}
  }
  \sum_{k=0}^{\infty}
  q^{3k^2 - k - 3kn}
  \frac{
    \pqrfac{q}{n}{q}
  }{
    \pqrfac{q}{2n-3k+1}{q}\,
    \pqrfac{q}{3k-n-1}{q}
  }.
\end{multline}
Next we consider the three inner sums on $k$. 
Denote them by $S_1(n)$, $S_2(n)$ and $S_3(n)$, respectively. 
They can be written as:
\begin{subequations}
\begin{align}
S_1(n) &=
  \sum_{k\in \mathbb{Z}} 
  q^{3k^2 - k - 3kn}
  \qbin{n}{3k-n}
   \\
S_2(n) &= 
   q^{-n}\sum_{k\in \mathbb{Z}}
  q^{3k^2 + k - 3kn}
  \qbin{n}{3k+1-n}
  \\
   S_3(n) &= 
   \sum_{k\in \mathbb{Z}}
  q^{3k^2 - k - 3kn}
  \qbin{n}{3k-1-n}
  .
  \end{align}
  \end{subequations}
Note that these sums run over $k\in \mathbb{Z}$. There are natural limits in all these sums,  and they are actually finite sums. We will show the following:
\begin{proposition}\label{prop:1} Let $S_1(n)$, $S_2(n)$ and $S_3(n)$ be as above. Then:
\begin{enumerate}[(i.)]
\item $S_2(n) = S_3(n)$.
\item $S_1(n) - S_3(n) = (-1)^n q^{-\binom{n+1}{2}}.$
\end{enumerate}
\end{proposition}
Using Proposition~\ref{prop:1}, we see that \eqref{eqn9} reduces to:
\begin{multline*}
\sum_{n=0}^{\infty}
  q^{n^2 } z^n
  \frac{
    \pqrfac{q}{3m}{q}
  }{
    \pqrfac{q}{3m-n}{q}\,
    \pqrfac{q}{n}{q}
  }
(S_1(n)+\omega S_2(n)+ \omega^2 S_3(n) \\
=
\sum_{n=0}^{\infty}
  q^{n^2 } z^n
  \frac{
    \pqrfac{q}{3m}{q}
  }{
    \pqrfac{q}{3m-n}{q}\,
    \pqrfac{q}{n}{q}
  } (-1)^n q^{-\binom{n+1}2} \\
  +
\sum_{n=0}^{\infty}
  q^{n^2 } z^n
  \frac{
    \pqrfac{q}{3m}{q}
  }{
    \pqrfac{q}{3m-n}{q}\,
    \pqrfac{q}{n}{q}
  } S_3(n) (1+\omega+\omega^2).
  \end{multline*}
  Since $1+\omega+\omega^2=0, $ the second sum is $0$. In fact, the first sum can be factored. We have:
\begin{equation}\label{sum4}
\sum_{n=0}^{\infty}
  q^{n^2 } z^n
  \frac{
    \pqrfac{q}{3m}{q}
  }{
    \pqrfac{q}{3m-n}{q}\,
    \pqrfac{q}{n}{q}
  }
  (-1)^n q^{-\binom{n+1}2}
 = \pqrfac{z}{3m}{q}.
 \end{equation}
This is the $m\mapsto 3m$ 
case of \eqref{q-vandermonde1a}.


To complete the proof of \eqref{conj1-gen}, we still need to prove Proposition~\ref{prop:1}.

\begin{proof}[Proof of Proposition~\ref{prop:1}] 
We first prove part (i). Let
$$F(n,k) := q^{3k^2+k-3nk-n} \qbin{n}{3k+1-n}.$$ 
Note that $F(n,k)$ is the summand of $S_2(n)$.
Consider $F(n, n-k)$:
\begin{multline*}
F(n,n-k) = q^{3(n-k)^2+(n-k)- 3 n(n-k) - n} \qbin{n}{3(n-k)+1-n} \\
= q^{3k^2-k-3nk} \qbin{n}{2n-3k+1} = q^{3k^2-k-3nk} \qbin{n}{3k-1-n}.
\end{multline*}
This is the summand of $S_3(n)$. Thus each term of $S_2(n)$ corresponds to a term in $S_3(n)$ and the two sums are the same; the symmetry shows that  $S_3(n)$ is obtained by reversing the sum $S_2(n)$. This proves (i).

For the proof of part (ii), consider
$$g(n):= S_1(n)-S_3(n).$$
We observe that $g(0)= 1$ and show that
$$g(n)=-q^{-n} g(n-1).$$
We require the identities ~\cite[eq.\ (I.45)]{GR90}:
\begin{equation*}
\left[\begin{matrix} n \\ r \end{matrix}\right]_q = 
q^{r}\left[\begin{matrix} n-1 \\ r \end{matrix}\right]_q
+
\left[\begin{matrix} n-1 \\ r-1 \end{matrix}\right]_q 
=
\left[\begin{matrix} n-1 \\ r \end{matrix}\right]_q
+
q^{\,n-r}\left[\begin{matrix} n-1 \\ r-1 \end{matrix}\right]_q.
\end{equation*}

We have:
\begin{align*}
g(n) &=
\sum_{k\in\mathbb Z}
q^{3k^2-k -3nk}
\bigg( \qbin{n}{3k-n}
- 
\qbin{n}{3k-1-n} 
\bigg)
\\
&=
\sum_{k\in\mathbb Z}
q^{3k^2-k -3nk}
\Bigg(
\bigg(
q^{3k-n} \qbin{n-1}{3k-n} 
+
\qbin{n-1}{3k-n-1}
\bigg)
- \\
&\hspace{1.5in}
\bigg(
\qbin{n-1}{3k-1-n}
+
q^{2n-3k+1} \qbin{n-1}{3k-2-n}
\bigg)
\Bigg)
\\
&= 
\sum_{k\in\mathbb Z}
q^{3k^2-k -3nk}
\bigg(
q^{3k-n} \qbin{n-1}{3k-n} 
-
q^{2n-3k+1} \qbin{n-1}{3k-2-n}
\bigg)\\
&=
\sum_{k\in\mathbb Z}
q^{3k^2-k -3nk}
\bigg(
q^{3k-n} \qbin{n-1}{3k-1-(n-1)} 
-
q^{2n-3k+1} \qbin{n-1}{3k-3-(n-1)}
\bigg)\\
&=
\sum_{k\in\mathbb Z}
q^{3k^2-k -3nk+3k-n}
\qbin{n-1}{3k-1-(n-1)} \\
&\hspace{1.5in}
-
\sum_{k\in\mathbb Z}
q^{3(k+1)^2-(k+1) -3n(k+1)+2n-3(k+1)+1}
\qbin{n-1}{3k-(n-1)}
 \\
&= -q^{-n} g(n-1).
\end{align*}
In the last but one equality, we shifted the index of the second sum. 

From $g(n)=-q^{-n} g(n-1)$, we obtain, by iteration:
$$g(n) = (-1)^n q^{-\binom{n+1}{2}} g(0) =  (-1)^n q^{-\binom{n+1}{2}}.$$ 
\end{proof}
This completes the proof of \eqref{conj1-gen}, and thus also of Conjecture~\ref{conj:1}.

\begin{proof}[Proof of Theorem~\ref{th:conj1-gen2}] The proof is virtually the same as that of \eqref{conj1-gen} . We break the inner sum into three parts, invoke Proposition~\ref{prop:1}, and apply \eqref{q-vandermonde1a} to complete the proof.
\end{proof}

\begin{thebibliography}{1}

\bibitem{MR2339282}
M.~Aigner.
\newblock {\em A Course in Enumeration}, volume 238 of {\em Graduate Texts in
  Mathematics}.
\newblock Springer, Berlin, 2007.

\bibitem{Andrews1981}
G.~E. Andrews.
\newblock Ramanujan's ``lost'' notebook. {I}. {P}artial {$\theta $}-functions.
\newblock {\em Adv. in Math.}, 41(2):137--172, 1981.

\bibitem{MR1141317}
E.~A. Bender, R.~Coley, D.~P. Robbins, and H.~Rumsey, Jr.
\newblock Enumeration of subspaces by dimension sequence.
\newblock {\em J. Combin. Theory Ser. A}, 59(1):1--11, 1992.

\bibitem{GR90}
G.~Gasper and M.~Rahman.
\newblock {\em Basic {H}ypergeometric {S}eries}, volume~96 of {\em Encyclopedia
  of Mathematics and its Applications}.
\newblock Cambridge University Press, Cambridge, second edition, 2004.
\newblock With a foreword by Richard Askey.

\bibitem{MR4555237}
A.~Prasad and S.~Ram.
\newblock Splitting subspaces and a finite field interpretation of the
  {T}ouchard-{R}iordan formula.
\newblock {\em European J. Combin.}, 110:Paper No. 103705, 11, 2023.

\bibitem{MR4797454}
A.~Prasad and S.~Ram.
\newblock Set partitions, tableaux, and subspace profiles under regular
  diagonal matrices.
\newblock {\em European J. Combin.}, 124:Paper No. 104060, 27, 2025.

\bibitem{ram2024subspaceprofilesfinitefields}
S.~Ram.
\newblock Subspace profiles over finite fields and $q$-whittaker expansions of
  symmetric functions, 2024.

\bibitem{MR366686}
J.~Riordan.
\newblock The distribution of crossings of chords joining pairs of {$2n$}
  points on a circle.
\newblock {\em Math. Comp.}, 29:215--222, 1975.

\bibitem{MR46325}
J.~Touchard.
\newblock Sur un probl\`eme de configurations et sur les fractions continues.
\newblock {\em Canad. J. Math.}, 4:2--25, 1952.

\end{thebibliography}

\end{document}